\documentclass[final]{siamltex}
\usepackage{amsfonts}
\usepackage{amsmath}
\usepackage[english]{babel}
\usepackage{subfigure}
\usepackage{graphicx}

\pdfoutput=1

\newcommand{\RR}{\mathbb{R}}
\newcommand{\ints}{\int\limits}
\newcommand{\lam}{\lambda}

\newcommand{\ep}{\varepsilon}
\newcommand{\de}{\partial}
\newcommand{\zi}{\zeta}
\newcommand{\OO}{\mathcal{O}}

\newcommand{\RE}{{\rm{Re}}}

\newtheorem{remark}[theorem]{{\it Remark}}

\title{Analytical results for 2-D non-rectilinear waveguides based on the Green's function}

\author{Giulio Ciraolo \thanks{Dipartimento di Matematica e
Applicazioni per l'Architettura, Universit\`a di Firenze, Piazza
Ghiberti 27, 50122 Firenze, Italy, ({\tt ciraolo@math.unifi.it}).}
\and Rolando Magnanini \thanks{Dipartimento di matematica U. Dini,
Universit\`a di Firenze, Viale Morgagni 67/A, 50134 Firenze, Italy,
({\tt magnanin@math.unifi.it}).}}

\begin{document}

\maketitle

\begin{abstract}
We consider the problem of wave propagation for a 2-D rectilinear
optical waveguide which presents some perturbation. We construct a
mathematical framework to study such a problem and prove the
existence of a solution for the case of small imperfections. Our
results are based on the knowledge of a Green's function for the
rectilinear case.
\end{abstract}



\pagestyle{myheadings} \thispagestyle{plain} \markboth{G. CIRAOLO
AND R. MAGNANINI}{ANALYTICAL RESULTS FOR NON-RECTILINEAR WAVEGUIDES}

\section{Introduction}
An optical waveguide is a dielectric structure which guides and
confines an optical signal along a desired path. Probably, the best
known example is the optical fiber, where the light signal is
confined in a cylindrical structure. Optical waveguides are largely
used in long distance communications, integrated optics and many
other applications.

In a rectilinear optical waveguide, the central region (the {\it
core}) is surrounded by a layer with a lower index of refraction
called {\it cladding}. A protective {\it jacket} covers the
cladding. The difference between the indices of refraction of core
and cladding makes possible to guide an optical signal and to
confine its energy in proximity of the core.

In recent years, the growing interest in optical integrated circuits
stimulated the study of waveguides with different geometries. In
fact, electromagnetic wave propagation along perturbed waveguides is
still continuing to be widely investigated because of its importance
in the design of optical devices, such as couplers, tapers,
gratings, bendings imperfections of structures and so on.

In this paper we propose an analytical approach to the study of
non-rectilinear waveguides. In particular, we will assume that the
waveguide is a small perturbation of a rectilinear one and, in such
a case, we prove a theorem which guarantees the existence of a
solution.

There are two relevant ways of modeling wave propagation in optical
waveguides. In {\it closed waveguides} one considers a tubular
neighbourhood of the core and imposes Dirichlet, Neumann or Robin
conditions on its boundary (see \cite{Ol} and references therein).
The use of these boundary conditions is efficient but somewhat
artificial, since it creates spurious waves reflected by the
interface jacket-cladding. In this paper we will study {\it open
waveguides}, i.e. we will assume that the cladding (or the jacket)
extends to infinity. This choice provides a more accurate model to
study the energy radiated outside the core (see \cite{SL} and
\cite{Ma}).

Thinking of an optical signal as a superposition of waves of
different frequency (the \emph{modes}), it is observed that in a
rectilinear waveguide most of the energy provided by the source
propagates as a finite number of such waves (the \emph{guided
modes}). The guided modes are mostly confined in the core; they
decay exponentially transversally to the waveguide's axis and
propagate along that axis without any significant loss of energy.
The rest of the energy (the \emph{radiating energy}) is made of
\emph{radiation} and \emph{evanescent} modes, according to their
different behaviour along the waveguide's axis (see \S \ref{section
preliminaries} for further details). The electromagnetic field can
be represented as a discrete sum of guided modes and a continuous
sum of radiation and evanescent modes.

As already mentioned, in this paper we shall present an analytical
approach to the study of time harmonic wave propagation in perturbed
2-D optical waveguides. As a model equation, we will use the
following \emph{Helmholtz equation} (or \emph{reduced wave
equation}):
\begin{equation} \label{helm}
\Delta u(x,z) + k^2 n(x,z)^2 u(x,z) = f(x,z),
\end{equation}
with $(x,z)\in\RR^2$, where $n(x,z)$ is the index of refraction of
the waveguide, $k$ is the wavenumber and $f$ is a function
representing a source. The axis of the waveguide is assumed to be
the $z$ axis, while $x$ denotes the transversal coordinate.

Our work is strictly connected to the results in \cite{MS}, where
the authors derived a resolution formula for \eqref{helm}, obtained
as a superposition of guided, radiation and evanescent modes, in the
case in which the function $n$ is of the form
\begin{equation}\label{n}
  n:=n_0 (x) = \begin{cases} n_{co}(x), & |x|\leq h,\\
  n_{cl}, & |x| > h, \end{cases}
\end{equation}
where $n_{co}$ is a bounded function decreasing along the positive
direction and $2h$ is the width of the core. Such a choice of $n$
corresponds to an index of refraction depending only on the
transversal coordinate and, thus, \eqref{helm} describes the
electromagnetic wave propagation in a rectilinear open waveguide. By
using the approach proposed in \cite{AC1}, the results in \cite{MS}
have been generalized in \cite{Ci1} to the case in which the index
of refraction is not necessarily decreasing along the positive
direction. The use of a rigorous transform theory guarantees that
the superposition of guided, radiation and evanescent modes is
complete. Such results are recalled in \S \ref{section
preliminaries}. The problem of studying the uniqueness of the
obtained solution and its outgoing nature will be addressed
elsewhere.

In this paper we shall study small perturbations of rectilinear
waveguides and present a mathematical framework which allows us to
study the problem of wave propagation in perturbed waveguides. In
particular, we shall assume that it is possible to find a
diffeomorphism of $\RR^2$ such that the non-rectilinear waveguide is
mapped in a rectilinear one. Thanks to our knowledge of a Green's
function for the rectilinear case, we are able to prove the
existence of a solution for small perturbations of 2-D rectilinear
waveguides by using the contraction mapping theorem.

In order to use such theorem, we shall prove that the inverse of the
operator obtained by linearizing the problem is continuous (see
Theorem \ref{teo esistenza}). Such a problem has been solved by
using weighted Sobolev spaces, which are commonly used when dealing
with Helmholtz equation (see, for instance, \cite{Le}).

In a forthcoming work, the results obtained in this paper will be
used to show several numerical results interesting for the
applications.

In \S \ref{section framework description} we describe our
mathematical framework for studying non-rectilinear waveguides.
Since our results are based on the knowledge of a Green's function
for rectilinear waveguides, in \S \ref{section preliminaries} we
recall the main results obtained in \cite{MS}.

Section \ref{section asympt lemmas} will be devoted to some
technical lemmas needed in \S \ref{Section proofs main theorems}.
The existence of a solution for the problem of perturbed waveguides
will be proven in Theorem \ref{teo esistenza}. Crucial to our
construction are the estimates contained in \S \ref{Section proofs
main theorems}, in particular the ones in Theorem \ref{lemma L0-1}.

Appendix \ref{section estimates RN} contains results on the global
regularity for solutions of the Helmholtz equation in $\RR^N$,
$N\geq 2$, that we need in Theorem \ref{teo esistenza}.

\section{Non-rectilinear waveguides: framework description}\label{section framework description}
When a rectilinear waveguide has some imperfection or the waveguide
slightly bends from the rectilinear position, we cannot assume that
its index of refraction $n$ depends only on the transversal
coordinate $x$. From the mathematical point of view, in this case,
we shall study the Helmholtz equation
\begin{equation}\label{helm xz ch3}
\Delta u + k^2 n_\ep(x,z)^2 u =f, \quad \textmd{in } \RR^2,
\end{equation}
where $n_\ep(x,z)$ is a perturbation of the function $n_0(x)$
defined in \eqref{n}, representing a ``perfect'' rectilinear
configuration.

We denote by $L_0$ and $L_\ep$ the Helmholtz operators corresponding
to $n_0(x)$ and $n_\ep(x,z)$ respectively:
\begin{equation}\label{L0 Lep}
L_0= \Delta + k^2 n_0 (x)^2,\quad  L_\ep= \Delta + k^2 n_\ep
(x,z)^2.
\end{equation}

In \cite{MS}, the authors found a resolution formula for
\begin{equation*}
  L_0 u = f,
\end{equation*}
i.e. they were able to write explicitly (in terms of a Green's
function) the operator $L_0^{-1}$ and then a solution of
\eqref{helm}. Now, we want to use $L_0^{-1}$ to write higher order
approximations of solutions of \eqref{helm xz ch3}, i.e. of
\begin{equation} \label{Lep u f ch3}
  L_\ep u = f.
\end{equation}

The existence of a solution of \eqref{Lep u f ch3} will be proven in
Theorem \ref{teo esistenza} by using a standard fixed point
argument: since \eqref{Lep u f ch3} is equivalent to
\begin{equation*}
  L_0 u = f + (L_0 - L_\ep) u,
\end{equation*}
then we have
\begin{equation*}
  u= L_0^{-1} f + \ep L_0^{-1} \left( \frac{L_0 - L_\ep}{\ep} \right) u.
\end{equation*}

Our goal is to find suitable function spaces on which $L_0^{-1}$ and
$\frac{L_0 - L_\ep}{\ep}$ are continuous; then, by choosing $\ep$
sufficiently small, the existence of a solution will follow by the
contraction mapping theorem.

It is clear that this procedure can be extended to more general
elliptic operators; in \S \ref{Section proofs main theorems} we will
provide the details.

\section{A Green's function for rectilinear waveguides} \label{section preliminaries}
In this section we recall the expression of the Green's formula
obtained by Magnanini and Santosa in \cite{MS} and generalized in
\cite{Ci1} to a non-symmetric index of refraction.

We look for solutions of the homogeneous equation associated to
\eqref{helm} in the form
\begin{equation*}
  u(x,z)=v(x,\lam) e^{ik\beta z};
\end{equation*}
$v(x,\lam)$ satisfies the associated eigenvalue problem for $v$:
\begin{equation}\label{eq v ch2}
  v'' + [\lam - q(x)] v = 0, \quad \textmd{in } \RR,
\end{equation}
with
\begin{equation} \label{n* lambda q}
n_*=\max_{\RR}{n},\quad \lam = k^2(n_*^2 - \beta^2),\quad q(x) = k^2
[ n_*^2 - n(x)^2 ].
\end{equation}
The solutions of \eqref{eq v ch2} can be written in the following
form
\begin{equation}\label{vj ch2}
v_j(x,\lam)=
  \begin{cases}
    \phi_j(h,\lam) \cos Q (x-h) + \frac{\phi_j'(h,\lam)}{Q} \sin Q(x-h), & \text{if } x > h, \\
    \phi_j(x,\lam), & \text{if } |x| \leq h, \\
    \phi_j(-h,\lam) \cos Q (x+h) + \frac{\phi_j'(-h,\lam)}{Q} \sin Q(x+h), & \text{if } x < -h, \\
  \end{cases}
\end{equation}
for $j=s,a$, with $Q=\sqrt{\lam-d^2}$, $d^2=k^2(n_*^2-n_{cl}^2)$ and
where the $\phi_j$'s are solutions of \eqref{eq v ch2} in the
interval $(-h,h)$ and satisfy the following conditions:
\begin{equation}\label{cond iniziali v ch2}
  \begin{array}{cc}
    \phi_s(0,\lam)=1, & \phi_s'(0,\lam)=0, \\
    \phi_a(0,\lam)=0, & \phi_a'(0,\lam)=\sqrt{\lam}.
  \end{array}
\end{equation}
The indices $j=s,a$ correspond to symmetric and antisymmetric
solutions, respectively.

\vspace{1em}

\begin{remark}\label{remark classification
solutions} (Classification of solutions).  {\rm The eigenvalue
problem \eqref{eq v ch2} leads to three different types of solutions
of \eqref{helm} of the form $u_\beta(x,z)=v(x,\lam) e^{ik\beta z}$.
\begin{itemize}
\item {\sl Guided modes}: $0 < \lam < d^2$. It exists a finite number
of eigenvalues $\lam_m^j$, $m=1,\ldots, M_j$, satisfying the
equations
\begin{equation*}
\sqrt{d^2-\lam} \; \phi_j(h,\lam) + \phi_j'(h,\lam) = 0, \quad
j\in\{ s,a \},
\end{equation*}
and corresponding eigenfunctions $v_j(x,\lam_m^j)$ which satisfy
\eqref{eq v ch2}. In this case, $v_j(x,\lam_m^j)$ decays
exponentially for $|x| > h$:
\begin{equation*}
v_j(x,\lam_m^j) = \begin{cases} \phi_j(h,\lam_m^j) e^{- \sqrt{d^2 -\lam_m^j} (x-h)}, & x > h, \\
\phi_j(x,\lam_m^j), & |x|\leq h,\\
\phi_j(-h,\lam_m^j) e^{\sqrt{d^2 -\lam_m^j} (x+h)}, & x < - h.
\end{cases}
\end{equation*}
In the $z$ direction, $u_\beta$ is bounded and oscillatory, because
$\beta$ is real.

\item {\sl Radiation modes}: $d^2 < \lam < k^2 n_*^2$. In this case,
$u_\beta$ is bounded and oscillatory both in the $x$ and $z$
directions.

\item {\sl Evanescent modes}: $ \lam > k^2 n_*^2$. The functions $v_j$
are bounded and oscillatory. In this case $\beta$ becomes imaginary
and hence $u_\beta$ decays exponentially in one direction along the
$z$-axis and increases exponentially in the other one.
\end{itemize}
}
\end{remark}

\vspace{1em}

By using the theory of Titchmarsh on eigenfunction expansions, we
can write a Green's function for \eqref{helm} as superposition of
guided, radiation and evanescent modes:
\begin{equation}\label{Green ch2}
  G(x,z;\xi,\zi) = \sum_{j\in\{s,a\}} \ints_{0}^{+\infty} \frac{e^{i|z-\zi| \sqrt{k^2n_*^2 - \lam}}}{2i
  \sqrt{k^2n_*^2 -\lam}} v_j(x,\lam) v_j(\xi,\lam) d\rho_j(\lam),
\end{equation}
with
\begin{equation*}
  \langle d\rho_{j}, \eta \rangle = \sum_{m=1}^{M_j} r_m^j \eta (\lam_m^j) + \frac{1}{2\pi} \ints_{d^2}^{+\infty}
  \frac{\sqrt{\lam -d^2}}{(\lam - d^2) \phi_j(h,\lam)^2 + \phi_j'(h,\lam)^2} \eta(\lam)
  d\lam,
\end{equation*}
for all $\eta\in C_0^\infty(\RR)$, where
\begin{equation*}
  r_m^j = \left[ \ints_{-\infty}^{+\infty} v_j(x,\lam_m^j)^2 dx \right]^{-1} =
  \frac{\sqrt{d^2-\lam_m^j}}{ \sqrt{d^2-\lam_m^j} \ints_{-h}^h \phi_j(x,\lam_m^j)^2 dx +
  \phi_j(h,\lam_m^j)^2}.
\end{equation*}
and where $v_j(x,\lam)$ are defined by \eqref{vj ch2} (see
\cite{Ci1} for further details).

We notice that \eqref{Green ch2} can be split up into three summands
\begin{equation*}
  G=G^g + G^r + G^e,
\end{equation*}
where
\begin{subequations}
\begin{equation}\label{G^g ch2}
 G^g(x,z;\xi,\zi) = \sum_{j\in \{ s,a \}} \sum_{m=1}^{M_j} \frac{e^{i|z-\zi|
 \sqrt{k^2 n_*^2 - \lam_m^j}}}{2i \sqrt{k^2n_*^2 -\lam_m^j}}
 v_j(x,\lam_m^j) v_j(\xi,\lam_m^j) r_m^j,
\end{equation}
\begin{equation}\label{G^r ch2}
 G^{r} (x,z;\xi,\zi) = \frac{1}{2\pi} \sum_{j\in\{s,a\}} \ints_{d^2}^{k^2 n_{*}^2}
 \frac{e^{i|z-\zi| \sqrt{k^2n_*^2 - \lam}}}{2i \sqrt{k^2n_*^2 -\lam}} v_j(x,\lam)
 v_j(\xi,\lam) \sigma_j(\lam) d\lam,
\end{equation}
\begin{equation}\label{G^e ch2}
 G^{e} (x,z;\xi,\zi) = -\frac{1}{2\pi} \sum_{j\in\{s,a\}} \ints_{k^2 n_{*}^2}^{+\infty}
 \frac{e^{-|z-\zi| \sqrt{\lam - k^2n_*^2}}}{2\sqrt{\lam - k^2n_*^2}}
 v_j(x,\lam) v_j(\xi,\lam) \sigma_j(\lam) d\lam,
\end{equation}
\end{subequations} with
\begin{equation}\label{sigma ch2}
\sigma_j(\lam) = \frac{\sqrt{\lam -d^2}}{(\lam - d^2)
\phi_j(h,\lam)^2 + \phi_j'(h,\lam)^2}.
\end{equation}
$G^g$ represents the guided part of the Green's function, which
describes the guided modes, i.e. the modes propagating mainly inside
the core; $G^r$ and $G^e$ are the parts of the Green's function
corresponding to the radiation and evanescent modes, respectively.
The radiation and evanescent components altogether form the
radiating part $G^{rad}$ of $G$:
\begin{equation}\label{G^rad ch2}
  G^{rad} = G^r + G^e = \frac{1}{2\pi}  \sum_{j\in\{s,a\}} \ints_{d^2}^{+\infty}
  \frac{e^{i|z-\zi| \sqrt{k^2n_*^2 - \lam}}}{2i \sqrt{k^2n_*^2 -\lam}}
  v_j(x,\lam) v_j(\xi,\lam) \sigma_j(\lam) d\lam.
\end{equation}

\vspace{2em}

\section{Asymptotic Lemmas} \label{section asympt lemmas}
This section contains some lemmas which will be useful in the rest
of the paper.

\vspace{1em}

\begin{lemma}\label{lemma v equilim}
Let $q\in L^1_{loc} (\RR)$ and $\lam_0=\min(\lam_1^s,\lam_1^a)$,
where $\lam_1^s$ and $\lam_1^a$ are defined in Remark \ref{remark
classification solutions}. Let $\phi_j$, $j\in \{s,a\}$, be defined
by \eqref{vj ch2}. Then, the following estimates hold for $x\in
[-h,h]$ and $\lam \geq \lam_0$:
\begin{equation}\label{phi phi' bounded}
|\phi_j(x,\lam)| \leq \Phi_*, \quad |\phi_j'(x,\lam)| \leq \Phi_*
\sqrt{\lam},
\end{equation}
where
\begin{equation}\label{PHI_M}
\Phi_* := \exp \left\{ \frac{1}{2\sqrt{\lam_0}} \ints_{-h}^h |q(t)|
dt \right\} .
\end{equation}
\end{lemma}

\begin{proof}
We consider the function
\begin{equation*}
\psi_j(x,\lam)= \phi_j'(x,\lam)^2 + \lam \phi_j(x,\lam)^2,
\end{equation*}
and notice that
\begin{equation*}
\psi_j'(x,\lam) = 2 q(x) \phi_j(x,\lam) \phi_j'(x,\lam),
\end{equation*}
as it follows from \eqref{eq v ch2}. By using Young's inequality, we
get that $\psi_j$ satisfies
\begin{equation*}
\begin{cases}
\displaystyle \psi_j'(x,\lam) \leq \frac{|q(x)|}{\sqrt{\lam}} \psi_j(x,\lam), & \\
\psi_j(0,\lam)=\lam.
\end{cases}
\end{equation*}
Therefore, by integrating the above inequality, we obtain that
\begin{equation*}
\psi_j(x,\lam) \leq \lam \exp \left\{ \frac{1}{\sqrt{\lam}}
\ints_{-h}^h |q(t)| dt \right\} \leq \lam \Phi_*^2,
\end{equation*}
which implies \eqref{phi phi' bounded}.
\end{proof}

\vspace{1em}

In the next two lemmas we study the asymptotic behaviour of the
function $\sigma_j(\lam)$ as $\lam \to +\infty $ and $\lam \to d^2$,
respectively.

\vspace{1em}

\begin{lemma} \label{corollary sigma infty}
Let $\sigma_j(\lam),\ j\in\{s,a\},$ be the quantities defined in
\eqref{sigma ch2}. The following asymptotic expansions hold as
$\lam\to\infty$:
\begin{equation}\label{rho infty}
\sigma_s(\lam) = \frac{1}{\sqrt{\lam -d^2}} +  \OO
\left(\frac{1}{\lam} \right),\quad \sigma_a(\lam) =
\frac{1}{\sqrt{\lam}} +  \OO \left(\frac{1}{\lam} \right).
\end{equation}
\end{lemma}

\begin{proof}
By multiplying
\begin{equation*}
    \phi_j''(x,\lam) + [\lam - q(x)] \phi_j (x,\lam) =0, \quad
    x\in [-h,h],
\end{equation*}
by $\phi_j'(x,\lam)$ and integrating in $x$ over $(0,h)$, we find
\begin{equation*}
    \phi_j'(h,\lam)^2 - \phi_j'(0,\lam)^2 + (\lam - d^2)
    [\phi_j(h,\lam)^2 - \phi_j(0,\lam)^2]  = 2 \ints_0^h [q(x) - d^2] \phi_j (x,\lam)
    \phi_j' (x,\lam) dx.
\end{equation*}
Thus, by using \eqref{cond iniziali v ch2}, we obtain the following
inequalities:
\begin{eqnarray*}
&& \big{|} \phi_s'(h,\lam)^2 + (\lam - d^2) \phi_s(h,\lam)^2 - (\lam
- d^2) \big{|}  \leq 2k^2(n_*^2 + n_{cl}^2) \ints_0^h |\phi_s
(x,\lam) \phi_s'(x,\lam)| dx, \\
&& \big{|} \phi_a'(h,\lam)^2 + (\lam - d^2) \phi_a(h,\lam)^2 - \lam
\big{|}  \leq 2k^2(n_*^2 + n_{cl}^2) \ints_0^h |\phi_a (x,\lam)
\phi_a'(x,\lam)| dx.
\end{eqnarray*}

The asymptotic formulas \eqref{rho infty} follow from the two
inequalities above, \eqref{sigma ch2} and the bounds \eqref{phi phi'
bounded} for $\phi_j(x,\lam)$ and $\phi_j'(x,\lam)$.
\end{proof}

\vspace{1em}

\begin{lemma} \label{lemma rho}
Let $\sigma_j(\lam),\ j\in\{s,a\},$ be the quantities defined in
\eqref{sigma ch2}. The following formulas hold for $\lam\to d^2$:
\begin{equation}\label{rho s d2}
\sigma_j(\lam)=
  \begin{cases}
    \displaystyle \frac{\sqrt{\lam-d^2}}{\phi_j'(h,d^2)^2} + \OO \left( \lam - d^2 \right),
     & \text{if } \phi_j'(h,d^2) \neq 0, \\
     & \\
    \displaystyle \frac{1}{\phi_j(h,d^2)^2 \sqrt{\lam-d^2}} + \OO \left( \sqrt{\lam - d^2} \right), &
    \text{otherwise}.
  \end{cases}
\end{equation}
\end{lemma}

\begin{proof}
We recall that, if $q\in L^1_{loc}(\RR)$, for $x\in[-h,h]$,
$\phi_s(x,\lam)$ and $\phi_a(x,\lam)$ are analytic in $\lam$ and
$\sqrt{\lam}$, respectively (see \cite{CL}). Thus, in a
neighbourhood of $\lam=d^2$, we write
\begin{equation*}
\begin{array}{cc}
  \phi(h,\lam) = \sum\limits_{m=0}^{+\infty} (\lam-d^2)^m a_m,
  & \phi'(h,\lam) = \sum\limits_{m=0}^{+\infty} (\lam-d^2)^m b_m, \\
  & \\
  \phi(h,\lam)^2 = \sum\limits_{m=0}^{+\infty} (\lam-d^2)^m \alpha_m,
  & \phi'(h,\lam)^2 = \sum\limits_{m=0}^{+\infty} (\lam-d^2)^m \beta_m,
\end{array}
\end{equation*}
where we omitted the dependence on $j$ to avoid too heavy notations.

We notice that $\alpha_0=a_0^2$, $\alpha_1=2a_0a_1$ and the same for
$b_0$ and $b_1$. From \eqref{sigma ch2} we have
\begin{equation}\label{rho s in lemma d2}
\begin{split}
\sigma_j(\lam)^{-1} & = \sqrt{\lam -d^2} \phi_j(h,\lam)^2 + \frac{1}{\sqrt{\lam-d^2}} \phi_j'(h,\lam)^2 \\
& = \frac{\beta_0}{\sqrt{\lam-d^2}} + \sum\limits_{m=0}^{+\infty}
(\lam-d^2)^{m+\frac{1}{2}} (\alpha_m + \beta_{m+1}).
\end{split}
\end{equation}
If $\beta_0\neq 0$, since $\beta_0=b_0^2$, \eqref{rho s d2} follows.
If $\beta_0=0$ we have that the leading term in \eqref{rho s in
lemma d2} is $\alpha_0 + \beta_1$. We notice that $\alpha_0\neq 0 $,
otherwise $\phi(x,d^2)\equiv 0$ for all $x\in\RR$. We know that
$\beta_1= 2 b_0 b_1=0$, because $b_0=0$. Then
$\alpha_0+\beta_1=\alpha_0$ and \eqref{rho s d2} follows.
\end{proof}

\vspace{2em}

\section{Existence of a solution} \label{Section proofs main theorems}
Let $\mu: \RR^2 \to \RR$ be a positive function. We will denote by
$L^2(\mu)$ the weighted space consisting of all the complex valued
measurable functions $u(x,z)$, $(x,z)\in\RR^2$, such that
\begin{equation*}
  \mu^{\frac{1}{2}} u \in L^2(\RR^2),
\end{equation*}
equipped with the natural norm
\begin{equation*}
  \|u\|_{L^2(\mu)}^2 = \ints_{\RR^2} |u(x,z)|^2 \mu(x,z) dx dz.
\end{equation*}

In a similar way we define the weighted Sobolev spaces $H^1(\mu)$
and $H^2(\mu)$. The norms in $H^1(\mu)$ and $H^2(\mu)$ are given
respectively by:
\begin{equation*}
 \|u\|_{H^1(\mu)}^2 = \ints_{\RR^2} |u(x,z)|^2 \mu(x,z) dx dz + \ints_{\RR^2} |\nabla u(x,z)|^2 \mu(x,z) dx dz ,
\end{equation*}
and
\begin{equation*}
 \|u\|_{H^2(\mu)}^2 = \ints_{\RR^2} |u(x,z)|^2 \mu(x,z) dx dz + \ints_{\RR^2} |\nabla u(x,z)|^2 \mu(x,z) dx dz  +
 \ints_{\RR^2} |\nabla^2 u(x,z)|^2 \mu(x,z) dx dz .
\end{equation*}
Here, $\nabla u$ and $\nabla^2 u$ denote the gradient and Hessian
matrix of $u$, respectively.

In this section we shall prove an existence theorem for the
solutions of \eqref{Lep u f ch3}. We will make use of results on
global regularity of the solution of \eqref{helm}; such results will
be proven in Appendix \ref{section estimates RN}.

The proofs in this section and in Appendix \ref{section estimates
RN} hold true whenever the (positive) weight $\mu$ has the following
properties:
\begin{equation}\label{mu1}
\begin{array}{c}
\mu \in C^2 (\RR^2) \cap L^1(\RR^2), \\
|\nabla \mu | \leq C_1 \mu,\quad |\nabla^2 \mu | \leq C_2 \mu,\ \
\textmd{ in } \RR^2,
\end{array}
\end{equation}
where $C_1$ and $C_2$ are positive constants.

In this section, for the sake of simplicity, we will assume that
$\mu$ is given by
\begin{equation}\label{mu equation}
  \mu (x,z) = \mu_1(x) \mu_2(z).
\end{equation}
with $\mu_j\in L^\infty(\RR) \cap L^1(\RR), \ j=1,2$. Analogous
results hold for every $\mu$ satisfying \eqref{mu1} and such that
\begin{equation}\label{proprieta mu separabile}
\mu(x,z)\leq \mu_1(x) \mu_2(z).
\end{equation}
For instance, it is easy to verify that the more commonly used
weight function $\mu(x,z)=(1+x^2+z^2)^{-a}$, $a>1$, satisfies
\eqref{mu1} and \eqref{proprieta mu separabile}.

Before starting with the estimates on $u$, we prove a preliminary
result on the boundness of the guided and radiated parts of the
Green's function.

\vspace{1em}

\begin{lemma} \label{lemma limitatezza G^g e G^r}
Let $G^g$ and $G^r$ be the functions defined in \eqref{G^g ch2} and
\eqref{G^r ch2}, respectively. Then
\begin{subequations}
\begin{equation}\label{Gg bounded ch3}
|G^g(x,z;\xi,\zi)| \leq \Phi_*^2 \sum_{j\in\{s,a\}} \sum_{m=1}^{M_j}
\frac{r_m^j}{2 \sqrt{k^2n_*^2 - \lam_m^j}} \, ,
\end{equation}
and
\begin{equation}\label{Gr bounded ch3}
|G^r (x,z;\xi,\zi)| \leq \max \left\{ \frac{1}{2},\ \frac{
\Phi_*}{4\sqrt{\pi}} \sum_{j\in\{s,a\}} \Upsilon_j ,\ \frac{\Phi_*^2
}{2\pi} \sum_{j\in\{s,a\}} \Upsilon_j^2 \right\}.
\end{equation}
\end{subequations}
Here,
\begin{equation*}
\Upsilon_j= \left( \ints_{d^2}^{k^2 n_*^2}
\frac{\sigma_j(\lam)}{2\sqrt{k^2n_*^2 - \lam}} d\lam
\right)^{\frac{1}{2}}, \quad j\in\{ s, a\},
\end{equation*}
where, as in Lemma \ref{lemma v equilim}, $\Phi_*$ is given by
\eqref{PHI_M}.
\end{lemma}

\begin{proof}
Since $G^g$ is a finite sum, from Remark \ref{remark classification
solutions} and Lemma \ref{lemma v equilim}, it is easy to deduce
\eqref{Gg bounded ch3}.

In the study of $G^r$ we have to distinguish three different cases,
according to whether $(x,z)$ and $(\xi,\zi)$ belong to the core or
not. Furthermore, we observe that $\Upsilon_j < +\infty$ as follows
form Lemma \ref{lemma rho}.

\underline{Case 1}: $x,\xi \in [-h,h]$. From Lemma \ref{lemma v
equilim} we have that $v_j(x,\lam)$ are bounded by $\Phi_*$. From
\eqref{G^r ch2} we have
\begin{equation*}
  |G^r| \leq \frac{1}{2\pi} \sum_{j\in\{s,a\}} \Phi_*^2 \Upsilon_j^2.
\end{equation*}

\underline{Case 2}: $|x|,|\xi| \geq h$. We can use the explicit
formula for $v_j$ (see \eqref{vj ch2}) and obtain by H\"{o}lder
inequality
\begin{equation*}
    |v_j(x,\lam)| \leq \sqrt{\phi_j(h,\lam)^2 + Q^{-2}
    \phi_j'(h,\lam)^2} = [Q\, \sigma_j(\lam)]^{-\frac{1}{2}}.
\end{equation*}
Therefore we have
\begin{equation*}
|G^r(x,z;\xi,\zi)| \leq \frac{1}{2\pi} \sum_{j\in\{s,a\}}
\ints_{d^2}^{k^2n_*^2} \frac{ d\lam }{ 2 \sqrt{\lam -d^2}
\sqrt{k^2n_*^2 -\lam}} = \frac{1}{2} ,
\end{equation*}
and hence \eqref{Gr bounded ch3} follows.

\underline{Case 3}: $|x|\leq h$ and $|\xi|\geq h$. We estimate
$|v_j(x,\lam)|$ by $\Phi_*$ and $v_j(\xi,\lam)$ by
$[Q\sigma_j(\lam)]^{-\frac{1}{2}}$, and write:
\begin{equation*}
\begin{split}
|G^r(x,z;\xi,\zi)|&  \leq  \frac{1}{2\pi} \Phi_* \sum_{j\in\{s,a\}}
\ints_{d^2}^{k^2n_*^2} \frac{1}{2\sqrt{k^2n_*^2 -\lam}}
\left[\frac{\sigma_j(\lam)}{Q} \right]^{\frac{1}{2}} d\lam \\
& \leq \frac{1}{4\pi} \Phi_* \sum_{j\in\{s,a\}} \Upsilon_j \left(
\ints_{d^2}^{k^2n_*^2} \frac{d\lam}{\sqrt{\lam -d^2} \sqrt{k^2n_*^2
-\lam}} \right)^{\frac{1}{2}}.
\end{split}
\end{equation*}
Again \eqref{Gr bounded ch3} follows.
\end{proof}

\vspace{1em}

In the next lemma we prove estimates that will be useful in Theorem
\ref{lemma L0-1}.

\vspace{1em}

\begin{lemma} \label{lemma p q}
Let $\mu_1 \in L^1(\RR)$ and $\mu_2 \in L^1(\RR) \cap L^2(\RR)$. For
$j,l \in \{s,a\}$ and $\lam,\eta \geq k^2n_*^2$, set
\begin{equation}\label{p jl}
  p_{j,l} (\lam,\eta) = \ints_{-\infty}^{+\infty} v_j(x,\lam) v_l(x,\eta) \mu_1(x) dx,
\end{equation}
and
\begin{equation}\label{q lam mu}
  q(\lam,\eta) = \ints_{-\infty}^{+\infty} (e_{\lam,\eta}\star \mu_2) (z)
  \mu_2(z) dz,
\end{equation}
where $ e_{\lam,\eta} (z) = e^{-|z| (\sqrt{\lam - k^2 n_*^2} +
\sqrt{\eta - k^2 n_*^2})}$.

Then $p_{j,l}(\lam,\eta) = 0$ for $j\neq l$,
\begin{equation}\label{pjj < }
p_{j,j}(\lam,\eta)^2 \sigma_j(\lam) \sigma_j(\eta) \leq 4
\|\mu_1\|_1^2 \Bigg( \Phi_*^2 \sigma_j(\lam) +
\frac{1}{\sqrt{\lam-d^2}} \Bigg) \Bigg( \Phi_*^2 \sigma_j(\eta) +
\frac{1}{\sqrt{\eta-d^2}} \Bigg),
\end{equation}
and
\begin{equation} \label{q < }
|q(\lam,\eta)| \leq \min \left( \|\mu_2\|_1^2, \frac{ \|\mu_2\|_2^2
}{\sqrt[4]{\lam-k^2n_*^2}\sqrt[4]{\eta-k^2n_*^2}} \right).
\end{equation}
\end{lemma}

\begin{proof}
Since $v_s$ and $v_a$ are, respectively, even and odd functions of
$x$, then $p_{j,l} ( \lam, \eta) = 0$ for $j\neq l$.

By H\"{o}lder inequality, $p_{j,j}(\lam,\eta)^2 \leq
p_{j,j}(\lam,\lam) p_{j,j}(\eta,\eta)$; also, by the formula
\eqref{vj ch2}, we obtain
\begin{equation*}
\begin{split}
|p_{j,j} (\lam,\lam)| & \leq 2 \ints_0^h |\phi_j(x,\lam)|^2
|\mu_1(x)| dx + 2 \ints_h^{+\infty} |v_j(x,\lam)|^2 |\mu_1(x)| dx \\
& \leq 2 \Phi_*^2 \ints_0^h |\mu_1(x)| dx + \frac{2}{\sqrt{\lam-d^2}
\sigma_j(\lam)} \ints_h^{+\infty} |\mu_1(x)| dx,
\end{split}
\end{equation*}
and hence \eqref{pjj < }.

Now we have to estimate $q(\lam,\eta)$. Firstly we observe that
\begin{equation*}
|q(\lam,\eta)| \leq \ints_{-\infty}^{+\infty} \mu_2(z) dz
\ints_{-\infty}^{+\infty} \mu_2(\zi) d\zi = \|\mu_2\|_1^2,
\end{equation*}
which proves part of \eqref{q < }. Furthermore, from Young's
inequality (see Theorem 4.2 in \cite{LL}) and the
arithmetic-geometric mean inequality, we have
\begin{equation*}
\begin{split}
|q(\lam,\eta)| \leq \|e_{\lam,\eta}\|_1 \|\mu_2\|_2^2 & = \frac{2
\|\mu_2\|_2^2}{\sqrt{\lam - k^2 n_*^2} + \sqrt{\eta
- k^2 n_*^2}} \\
& \leq \frac{\|\mu_2\|_2^2}{\sqrt[4]{\lam - k^2 n_*^2} \sqrt[4]{\eta
- k^2 n_*^2}},
\end{split}
\end{equation*}
which completes the proof.
\end{proof}

\vspace{2em}

\begin{theorem} \label{teo norma Green}
Let $G$ be the Green's function \eqref{Green ch2}. Then
\begin{equation}\label{norma Green}
  \| G \|_{L^2(\mu \times \mu)} < +\infty.
\end{equation}
\end{theorem}

\begin{proof}
We write $G= G^g + G^r + G^e$, as in \eqref{G^g ch2}-\eqref{G^e
ch2}, and use Minkowski inequality:
\begin{equation} \label{minkoswki per G}
\|G\|_{L^2(\mu \times \mu)} \leq \|G^g\|_{L^2(\mu \times \mu)} +
\|G^r\|_{L^2(\mu \times \mu)} + \|G^e\|_{L^2(\mu \times \mu)}.
\end{equation}
From Lemma \ref{lemma limitatezza G^g e G^r} and \eqref{mu1} it
follows that
\begin{equation}\label{norma G^g G^r finita}
  \|G^g\|_{L^2(\mu \times \mu)},\|G^r\|_{L^2(\mu \times \mu)} < + \infty.
\end{equation}
It remains to prove that $\|G^e\|_{L^2(\mu \times \mu)} < +\infty $.
From \eqref{mu equation} we have
\begin{equation*}
\begin{split}
\|G^e\|_{L^2(\mu \times \mu)}^2 & = \ints_{\RR^2 \times \RR^2}
|G^e(x,z;\xi,\zi)|^2 \mu_1(x) \mu_2(z) \mu_1(\xi) \mu_2(\zi) dx dz d\xi d\zi \\
& = \ints_{\RR^2 \times \RR^2} G^e(x,z;\xi,\zi)
\overline{G^e(x,z;\xi,\zi)} \mu_1(x) \mu_2(z) \mu_1(\xi) \mu_2(\zi)
dx dz d\xi d\zi;
\end{split}
\end{equation*}
hence, thanks to Lemma \ref{lemma p q}, the definition \eqref{G^e
ch2} of $G^e$ and Fubini's theorem, we obtain:
\begin{equation*}
\|G^e\|_{L^2(\mu \times \mu)}^2  \leq \frac{1}{16\pi^2}
\sum_{j,l\in\{s,a\}} \ints_{k^2n_*^2}^{+\infty}
\ints_{k^2n_*^2}^{+\infty} p_{j,l}(\lam,\eta)^2
\sigma_j(\lam)\sigma_l(\eta) \frac{q(\lam,\eta) \; d\lam
d\eta}{\sqrt{\lam - k^2 n_*^2}\sqrt{\eta - k^2 n_*^2}}.
\end{equation*}
The conclusion follows from Lemmas \ref{lemma p q}, \ref{corollary
sigma infty} and \ref{lemma rho}.
\end{proof}

\vspace{2em}

\begin{corollary}\label{lemma L0-1}
Let $u$ be the solution of \eqref{helm} given by
\begin{equation}\label{u ch2}
  u(x,z)= \ints_{\RR^2} G(x,z;\xi,\zi) f (\xi,\zi) d\xi d\zi, \ \ (x,z) \in \RR^2,
\end{equation}
with $G$ as in \eqref{Green ch2} and let $f\in L^2(\mu^{-1})$. Then
\begin{equation}\label{norma u}
  \| u \|_{H^2(\mu)} \leq C \| f \|_{L^2 (\mu^{-1})},
\end{equation}
where
\begin{equation} \label{C ch3}
C^2=\frac{5}{2} + 2C_2 + \left[ \frac{3}{2} + 4C_2 + 8C_2^2 + (1+ 4
C_2) k^2n_*^2 + 2k^4 n_*^4 \right] \|G\|_{L^2(\mu \times \mu)}^2.
\end{equation}
\end{corollary}

\begin{proof}
From \eqref{u ch2} and by using H\"older inequality, it follows that
\begin{equation*}
\begin{split}
  \|u\|_{L^2(\mu)}^2 & = \ints_{\RR^2} \Big| \ints_{\RR^2} G(x,z;\xi,\zi)
  f(\xi,\zi) d\xi d\zi \Big|^2 \mu(x,z) dx dz \\
  & \leq \|f\|_{L^2 (\mu^{-1})}^2 \ints_{\RR^2} \ints_{\RR^2}
  | G(x,z;\xi,\zi)|^2 \mu(\xi,\zi) \mu(x,z) d\xi d\zi dx dz\\
  & = \|f\|_{L^2 (\mu^{-1})}^2 \|G\|_{L^2(\mu \times \mu)}^2,
\end{split}
\end{equation*}
and then we obtain \eqref{norma u} and \eqref{C ch3} from Lemmas
\ref{lemma norme grad} and \ref{lemma norme der sec}.
\end{proof}

\vspace{2em}

\begin{remark} \label{remark Lep}
{\rm In the next theorem we shall prove the existence of a solution of
$L_\ep u = f$. It will be useful to assume in general that $L_\ep$
is of the form
\begin{equation}\label{Lep u aij}
L_\ep = \sum_{i,j=1}^2 a_{ij}^\ep \de_{ij}  +  \sum_{i=1}^2
b_{i}^\ep \de_{i}  + c^\ep.
\end{equation}
This choice of $L_\ep$ is motivated by our project to treat
non-rectilinear waveguides. Our idea is that of transforming a
non-rectilinear waveguide into a rectilinear one by a change of
variables $\Gamma:\RR^2 \to \RR^2$.

For this reason, we suppose that $\Gamma$ is a $C^2$ invertible
function:
\begin{equation*}
\Gamma (s,t) = (x(s,t),z(s,t)).
\end{equation*}
By setting $w(s,t)=u(x,z)$, a solution $u$ of \eqref{helm} is
converted into a solution $w$ of
\begin{equation} \label{Lep per w}
| \nabla s|^2 w_{ss} + |\nabla t|^2 w_{tt} + 2 \nabla s \cdot \nabla
t \; w_{st} + \Delta s \cdot w_s + \Delta t \cdot w_t + c(s,t)^2 w =
F(s,t),
\end{equation}
where $c(s,t)=k n(x(s,t),z(s,t))$ and $F(s,t)= f (x(s,t),z(s,t))$.

If our waveguide is a slight perturbation of a rectilinear one, we
may choose $\Gamma$ as a perturbation of the identity map,
\begin{equation*}
\Gamma (s,t) = ( s+ \ep \varphi(s,t), t + \ep \psi(s,t) ),
\end{equation*}
and obtain $L_\ep w = F $ from \eqref{Lep per w}, where $L_\ep$ is
given by \eqref{Lep u aij}, with
\begin{equation}\label{coeff L_ep}
a_{ij}^\ep = \delta_{ij} + \ep \tilde{a}_{ij}^\ep , \ i,j=1,2;\quad
b_i= \ep \tilde{b}_i^\ep, \ i=1,2;\quad c^\ep = k^2 n_0 (x)^2 + \ep
\tilde{c}^\ep;
\end{equation}
we also may assume that
\begin{equation}\label{K ch3}
\left[ \sum_{i,j=1}^2 (\tilde{a}_{ij}^\ep )^2 \right]^{\frac{1}{2}},
\ \left[ \sum_{i=1}^2 (\tilde{b}_{i}^\ep )^2 \right]^{\frac{1}{2}},\
|\tilde{c}^\ep| \leq K \mu \quad \textmd{in } \RR^2,
\end{equation}
for some constant $K$ independent of $\ep$. }
\end{remark}

\vspace{2em}

\begin{theorem} \label{teo esistenza}
Let $L_\ep$ be as in Remark \ref{remark Lep} and let $f\in
L^2(\mu^{-1})$. Then there exists a positive number $\ep_0$ such
that, for every $\ep\in (0,\ep_0)$, equation $L_\ep u = f$ admits a
(weak) solution $u^\ep \in H^2(\mu)$.
\end{theorem}

\begin{proof}
We write
\begin{equation}\label{4.90}
L_\ep = L_0 + \ep \tilde{L}_\ep;
\end{equation}
clearly, the coefficients of $\tilde{L_\ep}$ are
$\tilde{a}_{ij}^\ep$, $\tilde{b}_{i}^\ep$ and $\tilde{c}^\ep$
defined in \eqref{coeff L_ep}. We can write \eqref{4.90} as
\begin{equation*}
  u+ \ep L_0^{-1} \tilde{L}_\ep u = L_0^{-1} f;
\end{equation*}
$L_0^{-1} f$ is nothing else than the solution of \eqref{helm}
defined in \eqref{u ch2}.

We shall prove that $L_0^{-1} \tilde{L}_\ep$ maps $H^2(\mu)$
continuously into itself. In fact, for $u\in H^2(\mu)$, we easily
have:
\begin{equation*}
\begin{split}
\|\tilde{L}_\ep u \|_{L^2(\mu^{-1})}^2 & \leq \ints_{\RR^2} \Bigg[
\sum_{i,j=1}^2 (\tilde{a}_{ij}^\ep )^2 \sum_{i,j=1}^2 |u_{ij}|^2 +
\sum_{i=1}^2 (\tilde{b}_{i}^\ep )^2 \sum_{i=1}^2 |u_i|^2  +
(\tilde{c}^\ep)^2 |u|^2 \Bigg] \mu^{-1}
dx dz \\
& \leq K^2 \|u\|_{H^2(\mu)}^2.
\end{split}
\end{equation*}
Moreover, Corollary \ref{lemma L0-1} implies that
\begin{equation*}
    \| L_0^{-1} f \|_{H^2(\mu)} \leq C \|f\|_{L^2 (\mu^{-1})},
\end{equation*}
and hence
\begin{equation*}
\| L_0^{-1} \tilde{L}_\ep u \|_{H^2(\mu)} \leq CK \|u \|_{H^2(\mu)}.
\end{equation*}
Therefore, we choose $\ep_0= (CK)^{-1}$ so that, for $\ep\in
(0,\ep_0)$, the operator $\ep L_0^{-1} \tilde{L}_\ep$ is a
contraction and hence our conclusion follows from Picard's fixed
point theorem.
\end{proof}

\vspace{2em}

\section{Numerical results}\label{section numerical}
In this section we show how to apply our results to compute the
first order approximation of the solution of the perturbed problem.
The example presented here has only an illustrative scope; a more
extensive and rigorous description of the computational issues can
be found in \cite{Ci1} and \cite{Ci2}, where we apply our results to
real-life optical devices.

\begin{figure}[htpb]
\centering \subfigure[The perturbed waveguide. The dashed lines show
the effect of $\Gamma$ on the plane. In particular, they show how a
rectangular grid in the $(s,t)$-plane is mapped in the
$(x,z)$-plane.]{
\includegraphics[width=0.45\textwidth]{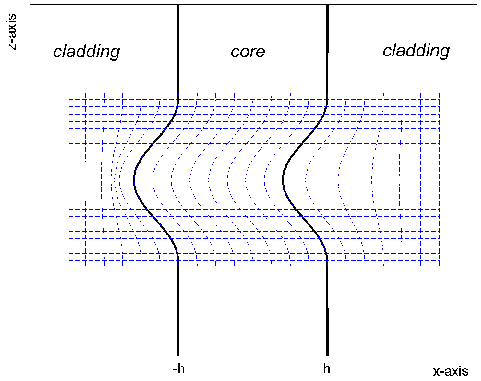}
 \label{Fig griglia} } \hfill \subfigure[Real part of $w^{(0)}$. Here, $w^{(0)}$
is a pure guided mode supported by the waveguide in the rectilinear
configuration.]{
\includegraphics[width=0.45\textwidth]{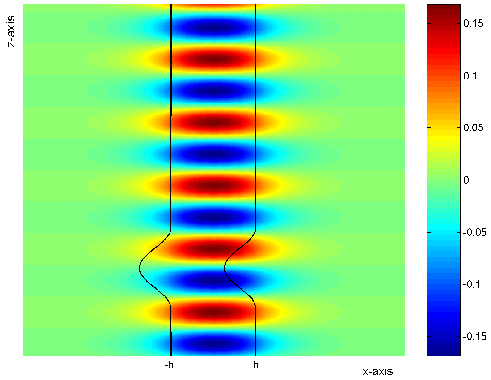} \label{Fig w0 real}
} \caption{The perturbed waveguide and the real part of $w^{(0)}$.}
\label{Fig1}
\end{figure}

In this section we study a perturbed slab waveguide as the one shown
in Fig.\ref{Fig griglia}. In the case of a rectilinear slab
waveguide it is possible to write the Green's formula explicitly and
numerically evaluate it (see \cite{MS}).

Having in mind the approach proposed in Remark \ref{remark Lep}, we
change the variables by using a $C^2-$function $\Gamma:\RR^2 \to
\RR^2$ of the following form:
\begin{equation*}
\Gamma (s,t) = (s, t + \ep S(s) T(t)),
\end{equation*}
where $S,T\in C_c^2(\RR)$; a good choice of $S$ and $T$ is
represented in Fig.\ref{Fig S e T}. In Fig.\ref{Fig griglia} we also
show how $\Gamma$ transforms the plane, by plotting in the
$(x,z)-$plane the image of a rectangular grid in the $(s,t)-$plane.

\begin{figure}
\centering \subfigure[The function $S$.]{
\includegraphics[width=0.4\textwidth]{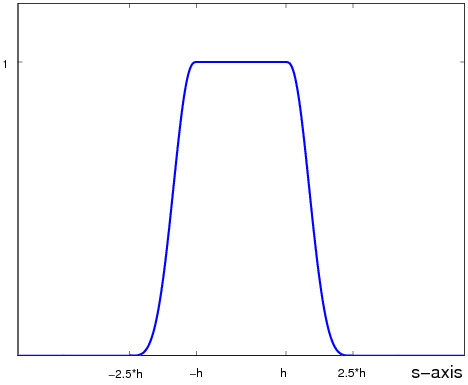} \label{Fig S}}
\hfill \subfigure[The function $T$.]{
\includegraphics[width=0.4\textwidth]{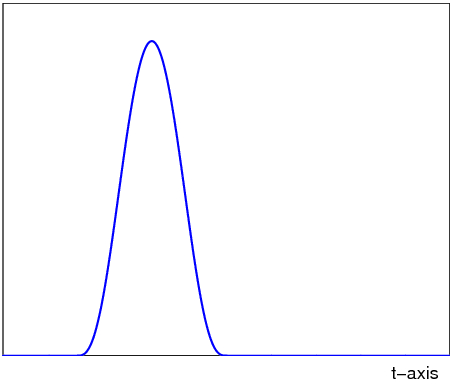} \label{Fig T}}

\caption{Our choice of the functions $S$ and $T$. Such a choice
corresponds to a perturbed waveguide as in Fig.\ref{Fig
griglia}.}\label{Fig S e T}

\end{figure}

By expanding $L_\ep$ and $w$ by their Neumann series, we find that
$w^{(0)}$ and $w^{(1)}$ (the zeroth and first order approximations
of $w$, respectively) satisfy
\begin{subequations} \label{w0 e w1}
\begin{equation}
\Delta w^{(0)} + k^2 n(s)^2 w^{(0)} = F(s,t),
\end{equation}
and
\begin{multline}
\Delta w^{(1)} + k^2 n(s)^2 w^{(1)} \\= -2 S'(s) T(t) w_{ss}^{(0)} -
2 S(s) T'(t) w_{st}^{(0)} - [S''(s) T(t) + S(s) T''(t)] w_{t}^{(0)},
\end{multline}
\end{subequations}
respectively.

In our simulations, we assume that $w^{(0)}$ is a pure guided mode
and calculate $w^{(1)}$ by using \eqref{w0 e w1} and the Green's
function \eqref{Green ch2}. In other words, we are taking a special
choice of $f$ and see what happens to the propagation of a pure
guided mode in the presence of an imperfection of the waveguide.

In Figures \ref{Fig w0 real}, \ref{Fig w1} and \ref{Fig w0 ep w1},
we set $k=5.0,\, h=0.2,\, n_{co}=2,\, n_{cl}=1$. With such
parameters, the waveguide supports two guided modes, corresponding
to the following values of the parameter $\lambda$:
$\lambda_1^s=23.7$ and $\lambda_1^a=73.5$.

As already mentioned, we are assuming that $w^{(0)}$ is a pure
guided mode. Here, $w^{(0)}$ is forward propagating and corresponds
to $\lambda_1^s$:
\begin{equation*}
w^{(0)}(s,t)=v_s(s,\lam_1^s) e^{i t \sqrt{k^2n_*^2 - \lam_1^s}};
\end{equation*}
the real part of $w^{(0)}$ is shown in Fig.\ref{Fig w0 real}.

Figures \ref{Fig w1 real} and \ref{Fig w1 abs} show the real part
and the absolute value of $w^{(1)}$, respectively. We do not write
here the numerical details of our computation and refer to
\cite{Ci1} and \cite{Ci2} for a more detailed description.

\begin{figure}
\centering \subfigure[Real part of $w^{(1)}$.]{
\includegraphics[width=0.45\textwidth]{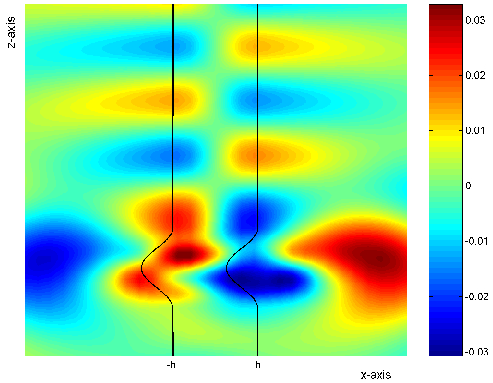}
\label{Fig w1 real}} \hfill \subfigure[Modulus of $w^{(1)}$.]{
\includegraphics[width=0.45\textwidth]{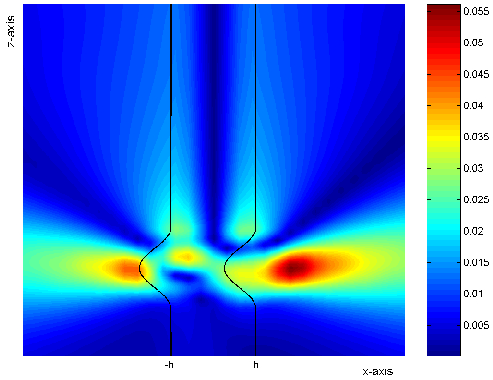}
\label{Fig w1 abs}} \caption{The real part and modulus of $w^{(1)}$
(the first order approximation of $w$).} \label{Fig w1}
\end{figure}

In Figures \ref{Fig w0 ep w1 real} and \ref{Fig w0 ep w1 abs} we
show the real part and the absolute value of $w^{(0)}+\ep w^{(1)}$,
respectively. Here, we choose $\ep=1$ to emphasize the effect of the
perturbation on the wave propagation. As is clear from Theorem
\ref{teo esistenza}, our existence result holds for $\ep\in
[0,\ep_0]$, where $\ep_0= (CK)^{-1}$ (which will be presumably less
than $1$). The computation of $\ep_0$ and the convergence of the
Neumann series related to $w$ have not been considered here; again,
we refer to \cite{Ci1} and \cite{Ci2} for a detailed study of such
issues.

\begin{figure}
\centering \subfigure[Real part of $w^{(0)} + \ep w^{(1)}$.]{
\includegraphics[width=0.45\textwidth]{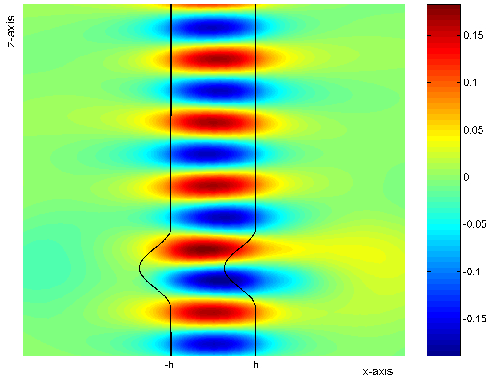}
\label{Fig w0 ep w1 real}} \hfill \subfigure[Modulus of $w^{(0)} +
\ep w^{(1)}$.]{
\includegraphics[width=0.45\textwidth]{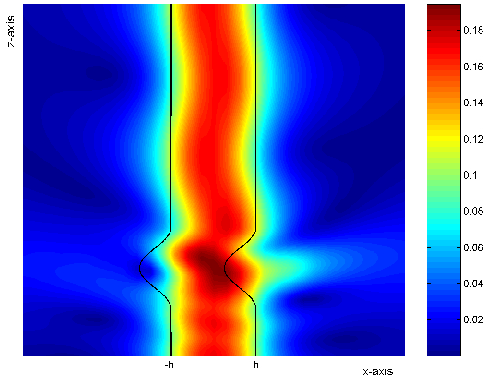}
\label{Fig w0 ep w1 abs}} \caption{The real part and modulus of
$w^{(0)} + \ep w^{(1)}$. The pictures clearly show the effect of a
perturbation of the waveguide: due to the presence of an
imperfection, the waveguide does not support the pure guided mode
$w^{(0)}$ and the other supported guided mode and the radiating
energy appear.} \label{Fig w0 ep w1}
\end{figure}


\section{Conclusions}
In this paper, we studied the electromagnetic wave propagation for
non-rectilinear waveguides, assuming that the waveguide is a small
perturbation of a rectilinear one. Thanks to the knowledge of a
Green's function for the rectilinear configuration, we provided a
mathematical framework by which the existence of a solution for the
scalar 2-D Helmholtz equation in the perturbed case is proven. Our
work is based on careful estimates in suitable weighted Sobolev
spaces which allow us to use a standard fix-point argument.

For the case of a slab waveguide (piecewise constant indices of
refraction), numerical examples were also presented. We showed that
our approach provide a method for evaluating how imperfections of
the waveguide affect the wave propagation of a pure guided mode.

In a forthcoming paper, we will address the computational issues
arising from the design of optical devices.

\vspace{2em}

\appendix\section{Regularity results}\label{section estimates RN} In this section we study the global
regularity of weak solutions of the Helmholtz equation. Since our
results hold in $\RR^N$, $N\geq2$, it will be useful to denote a
point in $\RR^N$ by $x$, i.e. $x=(x_1,\ldots,x_N) \in \RR^N$.

The results in this section can be found in literature in a more
general context for $N\geq 3$ (see \cite{Ag}). Here, under stronger
assumptions on $n$ and $\mu$, we provide an {\it ad hoc} treatment
that holds for $N\geq 2$.

We will suppose $f\in L^2 (\mu^{-1})$. Since $\mu$ is bounded, it is
clear that $f \in L^2 (\mu)$ too.

\vspace{1em}

\begin{lemma} \label{lemma norme grad}
Let $u\in H^1_{\textmd{loc}} (\RR^N)$ be a weak solution of
\begin{equation}\label{helmholtz R N}
\Delta u + k^2 n(x)^2 u = f , \quad x \in \RR^N,
\end{equation}
with $n\in L^\infty (\RR^N)$. Let $\mu$ satisfy the assumptions in
\eqref{mu1}. Then
\begin{equation}\label{norma gradiente u}
\ints_{\RR^N} |\nabla u|^2 \mu dx \leq \frac{1}{2} \ints_{\RR^N}
|f|^2 \mu dx + \left( 2C_2 + k^2 n_*^2 + \frac{1}{2} \right)
\ints_{\RR^N} |u|^2 \mu dx,
\end{equation}
where $n_*= \| n \|_{L^\infty (\RR^N)}$.
\end{lemma}

\begin{proof}
Let $\eta\in C_0^\infty(\RR^N)$ be such that
\begin{equation} \label{eta}
\eta(0)=1, \quad 0 \leq \eta \leq 1, \quad |\nabla \eta| \leq
1,\quad |\nabla^2 \eta| \leq 1,
\end{equation}
and consider the function defined by
\begin{equation}\label{mu_m def}
  \mu_m (x) = \mu(x) \eta \left( \frac{x}{m}\right).
\end{equation}
Then $\mu_m(x)$ increases with $m$ and converges to $\mu(x)$ as
$m\to +\infty$; furthermore
\begin{equation} \label{mu2}
\begin{array}{l}
 \displaystyle |\nabla \mu_m(x) | \leq  |\nabla \mu(x) | + \frac{1}{m} \mu(x) \leq \left(
 C_1 + \frac{1}{m} \right) \mu(x), \\
 \displaystyle |\nabla^2 \mu_m(x) | \leq  |\nabla^2 \mu(x) | + \frac{2}{m}  |\nabla \mu(x) | +
 \frac{1}{m^2} \mu(x) \leq \left( C_2 + \frac{2 C_1}{m} + \frac{1}{m^2} \right) \mu(x),
\end{array}
\end{equation}
for every $x \in \RR^N$.

Since $u$ is a weak solution of \eqref{helmholtz R N}, we have that
\begin{equation*}
  \ints_{\RR^N} \nabla u \cdot \nabla \phi dx - k^2 \ints_{\RR^N} n(x)^2 u \phi dx = - \ints_{\RR^N} f \phi dx,
\end{equation*}
for every $\phi \in H^1_{\textmd{loc}} (\RR^N)$. We choose $\phi =
\bar{u} \mu_m$ and obtain by Theorem 6.16 in \cite{LL}:
\begin{equation}\label{triangolo}
\ints_{\RR^N} |\nabla u|^2 \mu_m dx = - \ints_{\RR^N} \bar{u} \nabla
u \cdot \nabla \mu_m dx + k^2 \ints_{\RR^N} n(x)^2 |u|^2 \mu_m dx -
\ints_{\RR^N} f \bar{u} \mu_m dx.
\end{equation}
Integration by parts gives:
\begin{equation*}
   \RE \ints_{\RR^N} \bar{u} \nabla u \cdot \nabla \mu_m dx =
   - \frac{1}{2} \ints_{\RR^N} |u|^2 \Delta \mu_m dx;
\end{equation*}
hence, by considering the real part of \eqref{triangolo}, we obtain:
\begin{multline*}
\ints_{\RR^N} |\nabla u|^2 \mu_m dx \leq 2 \left(C_2 +
\frac{2C_1}{m} + \frac{1}{m^2} \right) \ints_{\RR^N} |u|^2 \mu dx +
k^2n_*^2 \ints_{\RR^N} |u|^2 \mu_m dx \\ + \left( \ \ints_{\RR^N}
|f|^2 \mu_m dx  \right)^{\frac{1}{2}} \left( \ \ints_{\RR^N} |u|^2
\mu_m dx  \right)^{\frac{1}{2}}  ;
\end{multline*}
here we have used \eqref{mu2} and H\"{o}lder inequality.

Young inequality and the fact that $\mu_m \leq \mu$ then yield:
\begin{equation*}
\ints_{\RR^N} |\nabla u|^2 \mu_m dx \leq \left[ 2\left(C_2 +
\frac{2C_1}{m} + \frac{1}{m^2} \right) + k^2 n_*^2 + \frac{1}{2}
\right] \ints_{\RR^N} |u|^2 \mu dx + \frac{1}{2}  \ints_{\RR^N}
|f|^2 \mu dx.
\end{equation*}
The conclusion then follows by the monotone convergence theorem.
\end{proof}

\begin{lemma} \label{lemma identita differenziale}
The following identity holds for every $u\in H^2_{\textmd{loc}}
(\RR^N)$ and every $\phi \in C_0^2 (\RR^N)$:
\begin{equation}\label{identita differenziale}
\ints_{\RR^N} |\Delta u|^2 \phi dx + \ints_{\RR^N} |\nabla u |^2
\Delta \phi dx = \ints_{\RR^N} |\nabla^2 u|^2 \phi dx + \RE
\ints_{\RR^N} (\nabla^2 \phi \nabla u ,\nabla u ) dx.
\end{equation}
\end{lemma}

\begin{proof}
It is obvious that, without loss of generality, we can assume that
$u\in C^3 (\RR^N)$; a standard approximation argument will then lead
to the conclusion.

For $u\in C^3 (\RR^N)$, \eqref{identita differenziale} follows by
integrating over $\RR^N$ the differential identity
\begin{multline*}
  \phi \sum_{i,j=1}^N u_{ii} \bar{u}_{jj} - \phi \sum_{i,j=1}^N u_{ij} \bar{u}_{ij} + \sum_{i,j=1}^N u_i \bar{u}_i
  \phi_{jj} - \RE \sum_{i,j=1}^N u_i \bar{u}_j \phi_{ij} \\
  = \RE \Bigg\{ \sum_{i,j=1}^N \left[ (\phi \bar{u}_j u_{ii})_j + (u_i \bar{u}_i \phi_j)_j -
  (\phi \bar{u}_j u_{ij})_i - (u_j \bar{u}_i \phi_j)_i  \right] \Bigg\},
\end{multline*}
and by divergence theorem.
\end{proof}

\begin{lemma}\label{lemma norme der sec}
Let $u\in H^1 (\mu)$ be a weak solution of \eqref{helmholtz R N}.
Then
\begin{multline}\label{norma deriv sec}
\ints_{\RR^N} |\nabla^2 u|^2 \mu dx \leq 2 \ints_{\RR^N} |f|^2 \mu
dx + 2k^4 n_*^4 \ints_{\RR^N} |u|^2 \mu dx + 4 C_2 \ints_{\RR^N}
|\nabla u |^2 \mu dx.
\end{multline}
where $C_2$ is the constant in \eqref{mu2}.
\end{lemma}

\begin{proof}
From well-known interior regularity results on elliptic equations
(see Theorem 8.8 in \cite{GT}), we have that if $u \in
H^1_{\textmd{loc}} (\RR^N)$ is a weak solution of \eqref{helmholtz R
N}, then $u \in H^2_{\textmd{loc}} (\RR^N)$. Then we can apply Lemma
\ref{lemma identita differenziale} to $u$ by choosing $\phi =
\mu_m$:
\begin{equation*}
\ints_{\RR^N} |\Delta u|^2 \mu_m dx + \ints_{\RR^N} |\nabla u |^2
\Delta \mu_m dx = \ints_{\RR^N} |\nabla^2 u|^2 \mu_m dx + \RE
\ints_{\RR^N} (\nabla^2 \mu_m \nabla u ,\nabla u ) dx.
\end{equation*}
From \eqref{helmholtz R N}, \eqref{mu2} and the above formula, we
have
\begin{equation*}
\begin{split}
\ints_{\RR^N} |\nabla^2 u|^2 \mu_m dx & = \ints_{\RR^N} |f - k^2
n(x)^2 u |^2 \mu_m dx + \ints_{\RR^N} |\nabla u
|^2 \Delta \mu_m dx - 2 \RE \ints_{\RR^N} (\nabla^2 \mu_m \nabla u, \nabla u ) dx \\
& \leq 2 \ints_{\RR^N} |f|^2 \mu_m dx + 2 k^4 n_*^4 \ints_{\RR^N} |u|^2 \mu_m dx \\
& \hspace{0.7cm} + 2 \left(C_2 + \frac{2C_1}{m} + \frac{1}{m^2}
\right) \ints_{\RR^N} |\nabla u |^2 \mu dx + 2 \ints_{\RR^N} |\nabla^2 \mu_m| |\nabla u|^2 dx \\
& \leq 2 \ints_{\RR^N} |f|^2 \mu_m dx + 2 k^4 n_*^4 \ints_{\RR^N} |u|^2 \mu_m dx \\
& \hspace{0.7cm} + 4 \left(C_2 + \frac{2C_1}{m} + \frac{1}{m^2}
\right) \ints_{\RR^N} |\nabla u |^2 \mu dx.
\end{split}
\end{equation*}
Since $\mu_m \leq \mu$, the proof is completed by taking the limit
as $m\to \infty$.
\end{proof}

\subsection*{Acknowledgments} Part of this work was written while
the first author was visiting the Institute of Mathematics and its
Applications (University of Minnesota). He wishes to thank the
Institute for the kind hospitality. The authors are also grateful to
Prof. Fadil Santosa (University of Minnesota) for several helpful
discussions.


\begin{thebibliography}{Z-Z-Z}

\bibitem[AC]{AC1} {\sc O.~Alexandrov and G.~Ciraolo},
  \emph{Wave propagation in a 3-D optical waveguide}.
  Math. Models Methods Appl. Sci. (M3AS), 14  (2004),  no. 6, pp.~819--852.

\bibitem[Ag]{Ag} {\sc  S.~Agmon}, \emph{Spectral Properties of
Schr\"{o}dinger Operators and Scattering Theory}, Ann. Sc. Norm.
Super. Pisa, Cl. Sci., IV. Ser. 2 (1975), pp.~151--218.

\bibitem[Ci1]{Ci1} {\sc G.~Ciraolo}, \emph{Non-rectilinear waveguides: analytical and
numerical results based on the Green's function}, PhD Thesis,
\verb+http://www.math.unifi.it/~ciraolo/+

\bibitem[Ci2]{Ci2} {\sc G.~Ciraolo}, \emph{A method of variation of boundaries for
waveguide grating couplers}, preprint,
\verb+http://www.math.unifi.it/~ciraolo/+

\bibitem[CL]{CL} {\sc E.~A. Coddington and N.~Levinson}, \emph{Theory of Ordinary
Differential Equations}, McGraw-Hill, New York, 1955.

\bibitem[GT]{GT} {\sc D.~Gilbarg and N.~S. Trudinger}, \emph{Elliptic partial differential equations of second
order}. Springer-Verlag, 1983.

\bibitem[Le]{Le} {\sc R.~Leis}, \emph{Initial boundary value problems in mathematical physics},
John Wiley, 1986.

\bibitem[LL]{LL} {\sc E.~H. Lieb and M.~Loss}, \emph{Analysis}, American Mathematical Society, Providence, RI, 1997.

\bibitem[LS]{LS} {\sc B.~M. Levitan and I.~S. Sargsjan}, \emph{Introduction to spectral theory :
selfadjoint ordinary differential operators}. Providence, R.I.,
American Mathematical Society, 1975.

\bibitem[Ma]{Ma} {\sc D.~Marcuse},
  \emph{Light Transmission Optics}, Van Nostrand Reinhold Company, New York, 1982.

\bibitem[MS]{MS} {\sc R.~Magnanini and F.~Santosa},
  \emph{Wave propagation in a 2-D optical waveguide}, \emph{SIAM J. Appl.
  Math.,} 61 (2001), pp.~1237--1252.

\bibitem[Ol]{Ol} {\sc A.~A. Oliner}, \emph{Historical perspectives on microwave field theory}, IEEE
Transactions on Microwave Theory and Techniques 32 (1984), no. 9,
pp.~1022--1045.

\bibitem[SL]{SL} {\sc A.~W. Snyder and D.~Love},
  \emph{Optical Waveguide Theory}, Chapman and Hall, London, 1974.

\bibitem[Ti]{Ti} {\sc E.~C. Titchmarsh},
  \emph{Eigenfunction expansions associated with second-order differential
  equations}. Oxford at the Clarendon Press, Oxford, 1946.

\end{thebibliography}
\end{document}